\documentclass{amsart}
\usepackage{changebar}
\usepackage{lineno}
\usepackage{amsthm}
\usepackage{verbatim}
\usepackage[utf8]{inputenc}
\usepackage{blindtext}
\usepackage{bm}

\makeatletter
\renewcommand\part{%
   \if@noskipsec \leavevmode \fi
   \par
   \addvspace{4ex}%
   \@afterindentfalse
   \secdef\@part\@spart}

\def\@part[#1]#2{%
    \ifnum \c@secnumdepth >\m@ne
      \refstepcounter{part}%
      \addcontentsline{toc}{part}{\thepart\hspace{1em}#1}%
    \else
      \addcontentsline{toc}{part}{#1}%
    \fi
    {\parindent \z@ \raggedright
     \interlinepenalty \@M
     \normalfont
     \ifnum \c@secnumdepth >\m@ne
       \Large\bfseries \partname\nobreakspace\thepart
       \par\nobreak
     \fi
     \huge \bfseries #2%
     \par}%
    \nobreak
    \vskip 3ex
    \@afterheading}
\def\@spart#1{%
    {\parindent \z@ \raggedright
     \interlinepenalty \@M
     \normalfont
     \huge \bfseries #1\par}%
     \nobreak
     \vskip 3ex
     \@afterheading}
\makeatother

\newtheorem{lemma}{Lemma}
\newtheorem{theorem}{Theorem}

\newcommand{\nmid}{\not \hspace{0.25em} \mid}

\begin{document}

\title[The diophantine equation $\left(2^{k}-1\right)\left(3^{k}-1\right)=x^{n} $  ] {The diophantine equation $\left(2^{k}-1\right)\left(3^{k}-1\right)=x^{n} $ } 

\author{Bo He} 
\address[Bo He]{1. Mathematisches Institut der Universit\"at 
G\"ottingen, Bunsenstr. 3-5, DE-37073, G\"ottingen, Germany; 2. Applied Mathematics insititute of Aba Teachers University, Wenchuan, Sichuan, 623002, P. R. China}
\email[Bo He]{bo.he01@stud.uni-goettingen.de; bhe@live.cn} 
\author{Chang Liu} 
\address[Chang Liu]{Mathematisches Institut der Universit\"at 
G\"ottingen, Bunsenstr. 3-5, DE-37073, G\"ottingen, Germany}
\email[Chang Liu]{chang.liu@mathematik.uni-goettingen.de} 

\date{ \today}
\vspace{0.3cm}
\begin{abstract}
In this paper, we investigate the Diophantine equation 
\[
(2^k - 1)(3^k - 1) = x^n
\]
and prove that it has no solution in positive integers $k, x, n > 2$.
\end{abstract}

\maketitle
\vspace*{-0.5cm}

\section{Introduction}
In 2000, Szalay \cite{1} studied the Diophantine equation
\[
(2^n - 1)(3^n - 1) = x^2
\]
and proved that it has no solution in positive integers \( n \) and \( x \). Additionally, he studied similar equations such as \( (2^n - 1)(5^n - 1) = x^2 \), demonstrating that they have only limited solutions with specific values of \( n \). In
the same year, Hajdu and Szalay \cite{2} showned that there is no solution for \((2^n - 1)(6^n - 1) = x^2\).

In 2001, Cohn's \cite{3} work focused on the Diophantine equation
$$
(a^n - 1)(b^n - 1) = x^2
$$
and explored the integer solutions for given values of \(a\) and \(b\), presenting general results and conjectured about the solvability of the equation for specific cases.

In 2002, Luca and Walsh \cite{4} applied a computational approach to solve completely the Cohn-type equations
    \[
    (a^k - 1)(b^k - 1) = x^2
    \]
    for nearly all pairs \((a, b)\) satisfying \(2 \leq b < a \leq 100\), leaving only 70 unresolved cases, without additional conditions. This significantly extends previous works on such equations. Furthermore, under the assumption of the ABC conjecture, they showed that a much stronger results can be derived, including the finiteness of solutions for equations like
    \[
    (x^m - 1)(y^n - 1) = z^2.
    \]

In 2007, Bennett \cite{5} investigated the Diophantine equation of the form
\[
(x^k - 1)(y^k - 1) = (z^k - 1)^t
\]
and focused on proving the finiteness of integer solutions, especially when \( t = 1 \) or \( t = 2 \) and \( k \) is sufficiently large. His approach involved using the hypergeometric method of Thue and Siegel, along with various gap principles, to show that there are no positive integer solutions for specific cases where \( z > 1 \) and \( k > 3 \).

\newpage

In this paper, we will prove the following theorem.

\begin{theorem}
The equation
\begin{equation}
\label{111}
\left(2^{k}-1\right)\left(3^{k}-1\right)=x^{n}
\end{equation}
has no solution in positive integers \( k, x, n>2\).
\end{theorem}

To facilitate our analysis, we can rewrite this equation as
\begin{equation}
\label{theorem}
\left(2^{k}-1\right)\left(3^{k}-1\right)=y^{q},
\end{equation}
where \( k, y\) are positive integers, and \( q \) is a odd prime number. Here $q$ denotes the least prime divisor of $n$ and $y=x^{n/q}$.       
It is noteworthy that the case of $q=2$ has been resolved by Szalay's work \cite{1}. Consequently, we only need to consider exclusively $q$ is an odd prime. 
\section{Proofs}
Let \( p \) be a prime and \( n \) a nonzero integer. The \( p \)-adic valuation, denoted by \( \nu_p(n) \), is defined as the exponent of \( p \) in the prime factorization of \( n \). The following result, known as the Lifting-the-Exponent (LTE) Lemma, is a fundamental tool in analyzing the \( p \)-adic properties of exponential expressions. 

\begin{lemma}[Lifting-the-exponent Lemma \cite{6}]
\label{Lifting-the-exponent Lemma}
Let \( p \) be a prime, and let \( a \) and \( b \) be integers such that \( k \) is a positive integer. Suppose \( p \mid (a - b) \) and \( p \nmid ab \). Then, the \( p \)-adic valuation \( \nu_p \) of \( a^k - b^k \) is given by
\[
\nu_p(a^k - b^k) = 
\begin{cases}
\nu_p(a - b) + \nu_p(k), & \text{if } p \text{ is odd}, \\
\nu_2(a - b) , & \text{if } p = 2 \text{ and } k \text{ is odd}, \\
\nu_2(a^2 - b^2) + \nu_2\left(\frac{k}{2}\right), & \text{if } p = 2 \text{ and } k \text{ is even}.
\end{cases}
\]
\end{lemma}

\begin{lemma}
\label{lemma k}
  Assume that $(k, y, q)$ is a positive integer solution where $q$ is a prime of equation \eqref{theorem}. We have

\begin{enumerate}
  \item $\nu_{2}(k) \geq q-2$,
  \item $\nu_{p}(k) \geq q-1$ for $p \le q$ with $p\in \{3, 5, 7\}$. 
\end{enumerate}  
\end{lemma}

\begin{proof}

For point (1), since $2 \mid (3^k - 1)$, it follows that  $2 \mid y$. Thus, we have
    \[
    3^k - 1 \equiv 0 \pmod{2^q},
    \]
    which implies that \( k \) must be even. By applying Lemma~\ref{Lifting-the-exponent Lemma}, we obtain
    \[
    q \leq \nu_2\left(y^q\right) = \nu_2\left(3^k - 1\right) = \nu_2\left(3^2 - 1\right) + \nu_2\left(\frac{k}{2}\right).
    \]
    Therefore, it follows that
    \begin{equation}\label{eq:2div}
    2^{q - 2} \mid k.
    \end{equation}

When $p=3$ in point (2), the condition \( 2^k - 1 \equiv 0 \pmod{3} \) implies that  $3 \mid y$. Thus, we have
    \[
    2^k - 1 \equiv 0 \pmod{3^q}.
    \]
    Applying Lemma~\ref{Lifting-the-exponent Lemma} again, we obtain
    \begin{equation}\label{eq:3div}
    3^{q - 1} \mid  k.
    \end{equation}
    
When $p=5$, assume that \( q \geq 5 \). Given that $4 \mid k$ from \eqref{eq:2div}, we have
    \[
    2^k - 1 \equiv 3^k - 1 \equiv 0 \pmod{5}.
    \]
    Additionally, by Lemma~\ref{Lifting-the-exponent Lemma}, we get
    \[
    \nu_5\left(2^k - 1\right) = \nu_5\left(3^k - 1\right) = 1 + \nu_5(k).
    \]
    Since \( q \) is odd, if \( q \) divides \( 2\left(1 + \nu_5(k)\right) \), it must also divide \( 1 + \nu_5(k) \). Thus, we conclude
    \[
    5^{q - 1} \mid k.
    \]

Finally, in the case $p=7$, from \eqref{eq:2div} and \eqref{eq:3div} we know that $6|k$. Then we have 
  \[
   2^k - 1 \equiv 3^k - 1 \equiv 0 \pmod{7}.
    \]
By Lemma~\ref{Lifting-the-exponent Lemma}, we get 
\[
   \nu_7\left(2^k - 1\right) = \nu_7 \left(3^k - 1\right) = 1 + \nu_7(k).
\]
The fact $q \mid 2(1+\nu_7(k))$ implies $q \mid (1+\nu_7(k))$. This completes the proof.  
\end{proof}

The following lemma addresses the divisibility properties of \( k \) for larger prime numbers.

\begin{lemma}
\label{key lemma}
Assume that \( (k, y, q) \) is a positive integer solution of equation \eqref{theorem}. Let $p$ be a prime such that \( p\ge 11\). If \( (p - 1)\mid k \), then
$$
\nu_{p}(k)>\frac{1}{2}\left(q- \frac{\log 6}{\log p}\cdot (p-1)\right).
$$    
In particular, if we further assume that $p\le q$, then $p \mid k$.  
\end{lemma} 

\begin{proof}
  Let us start at the analyzing of \( p \)-adic valuation on both sides of equation \eqref{theorem}. We assumed \( p - 1 \) divides \( k \). Then by Fermat's Little Theorem, it follows that \( p \) divides both \( a^{p-1} - 1 \) and \( a^k - 1 \) for any integer \( a \) satisfing \( p \nmid a \). By applying Lemma \ref{Lifting-the-exponent Lemma}, we obtain
\[
\nu_p\left(2^k - 1\right) = \nu_p\left(2^{p-1} - 1\right) + \nu_p\left(\frac{k}{p-1}\right) = \nu_p\left(2^{p-1} - 1\right) + \nu_p(k),
\]
and similarly,
\[
\nu_p\left(3^k - 1\right) = \nu_p\left(3^{p-1} - 1\right) + \nu_p(k).
\]
Next, we observe that
\[
\nu_p\left(2^{p-1} - 1\right) < \log 2 \cdot \frac{p - 1}{\log p}, \quad \nu_p\left(3^{p-1} - 1\right) < \log 3 \cdot \frac{p - 1}{\log p}.
\]
Thus, we conclude
\[
\nu_p\left(\left(2^k - 1\right)\left(3^k - 1\right)\right) < \log 6 \cdot \frac{p - 1}{\log p} + 2 \nu_p(k).
\]
On the right-hand side, since \( p \mid y \), we have
\[
\nu_p\left(y^q\right) \geq q.
\]
Combining the inequalities derived above, we obtain
\[
q < \log 6 \cdot \frac{p - 1}{\log p} + 2 \nu_p(k).
\]

Assume that $p\le q$ now. We have 
$$
\nu_{p}(k)>\frac{1}{2}\left(q- \frac{\log 6}{\log p} \cdot (p-1)\right) \ge \frac{1}{2}\left(q- \frac{\log 6}{\log 11} \cdot (q-1)\right)$$
$$
>\frac{1}{2}\left(q - \frac{3}{4}(q-1)\right) = \frac{q}{8}+ \frac{3}{8} >1.
$$
Which completes the proof of the lemma.
\end{proof}

\begin{theorem}\label{thm:2}
Assume that \( (k, y, q) \) is a positive integer solution to equation \eqref{theorem}. Then, it follows that \( q \mid k \).   
\end{theorem} 

\begin{proof}
    If \( q \in \{3, 5, 7\}\), the result follows directly from Lemma~\ref{lemma k}. Now, let us assume \( q \geq 11 \). Define \( P = \{ 2, 3, 5, 7, p_1, p_2, \ldots, p_n \} \) as the set of all primes less than or equal to \( q \),  where \( p_1 = 11 < p_2 < \cdots < p_n = q \).
 To prove the result, we need to show that \( (p - 1) \mid k \) for each \( p \in P \) using Lemma~\ref{key lemma}.

Now, assume that for some \( i \), we have \( (p_i-1) \mid k \). Our goal is to show that \( (p_{i+1}-1) \mid k \). For \( p = p_{i+1} \), it suffices to establish that
\begin{equation}\label{2}
    \nu_r(p - 1) \leq \nu_r(k)
\end{equation}
for all \( r = 2, 3, 5, 7, p_1, \ldots, p_i \).

We begin by considering the case where \( r = 2, 3, 5, 7\). By applying Lemma~\ref{lemma k} and utilizing the results previously derived, we find that
\[
\nu_r(p - 1) \leq \frac{\log (p - 1)}{\log r} \le \frac{\log (q-1)}{\log 2}  \le q - 2 \leq \nu_r(k).
\]

Next, consider the case where \( r \geq 11 \). We analyze two subcases based on the relationship between \( r \) and \( q^{1/2} \).

$\bullet$ If \( r > q^{1/2} \), it follows that
   \[
   \nu_r(p - 1) \le \frac{\log(p-1)}{\log r} < \frac{\log q}{\log r} < \frac{\log r^2}{\log r} = 2.
   \]
   Indeed, from Lemma \ref{key lemma}, we deduce that $r \mid k$.
   Thus, we conclude
   \[
   \nu_r(p - 1) \leq 1 \leq \nu_r(k).
   \]

$\bullet$ If \( r < q^{1/2} \), we claim that
   \[
   \frac{1}{2} \left(q - \frac{\log 6}{\log r} \cdot (r - 1)\right) \geq \frac{\log q}{\log r}.
   \]
   To prove this, suppose the contrary
   \[
 q - \frac{\log 6}{\log r} \cdot (r - 1)< \frac{2 \log q}{\log r}.
   \]
   Since \( q \geq 11 \), this would imply
   \[
   q < 2 \log 6 \cdot \frac{q^{1/2} - 1}{\log q} + \frac{2 \log q}{\log 11},
   \]
   a contradiction. We also have $\nu_r(k) > \frac{\log q}{ \log r} > \nu_r(p-1)$ by Lemma~\ref{key lemma}.

Therefore, by establishing inequality (\ref{2}), we have shown that \( (p_{i+1} - 1) \mid k \). Consequently, by applying Lemma~\ref{key lemma}, we deduce \( p_{i+1} \mid k \). Thus, the induction is complete, proving that each prime in the set \( P \) divides \( k \), including \( q \). This completes the proof.

\end{proof}
\newpage
We now prove the Theorem 1.
\begin{proof}

Let \( (k, y, q) \) be a positive integer solution to equation \eqref{theorem}. By Theorem~\ref{thm:2}, we have \( q \mid k \). Define that
\[
(X, Y, Z) := \left(2^{k / q}, 3^{k / q}, y\right).
\]
Then, \( (X, Y, Z) \) forms a positive integer solution to the Diophantine equation
\[
\left(X^q - 1\right)\left(Y^q - 1\right) = Z^q, \quad 0 < X < Y. 
\]
Let
\[
a:= X^q-1=2^k-1, \quad b:= Y^q-1=3^k-1, \quad ab = Z^q.
\]
There exists a positive integer \( t \) such that
\[
XY = Z + t.
\]
Expanding the expression
\[
\left((ab)^{1 / q} + t\right)^q = (a + 1)(b + 1),
\]
we obtain
\begin{equation}\label{eq:last}
q (ab)^{(q-1) / q} t + \binom{q}{2} (ab)^{(q-2) / q} t^2 + \cdots + t^q = a + b + 1.
\end{equation}
With $k\ge q \ge 3$, one can check that $a<b<a^2$ by the fact 
\[
b = 3^k-1 < 4^k - 2^{k+1} +1 =(2^k-1)^2 =a^2.
\] 
Then we get 
\[
q(ab)^{(q-1)/q} t \ge 3(ab)^{2/3} >3b > a + b +1. 
\]
This contradicts to \eqref{eq:last}. Therefore, no such positive integer solution exists, which completes the proof.
 \end{proof} 

\section*{Acknowledgements} The authors express their gratitude to Prof. Preda Mih\u{a}ilescu for his help during the early preparation of this paper. The authors also thank Yoshinosuke Hirakawa and the anonymous referee for the constructive suggestions to improve an earlier draft of this paper. The first author was supported by Natural Science Foundation of China (Grant No. 12161001). The second author acknowledges the support of the China Scholarship Council program (Project ID: 202106310023). 
 
\end{document}